\documentclass[11pt]{amsart}
\usepackage[utf8]{inputenc}
\usepackage{enumerate,amssymb,xcolor,comment}
\usepackage[footskip=1cm, headheight = 16pt, top=3.7cm, bottom=3cm,  right=2.5cm,  left=2.5cm ]{geometry}
\usepackage[colorlinks,linkcolor=blue,citecolor=red]{hyperref}
\usepackage{todonotes}

\newtheorem{question}{Question}[section]
\newtheorem{lemma}[question]{Lemma}
\newtheorem{theorem}[question]{Theorem}

\newtheorem{corollary}[question]{Corollary}

\usepackage[foot]{amsaddr}

\usepackage[T1]{fontenc}

\usepackage[leqno]{amsmath}

\makeatletter
\newcommand{\leqnomode}{\tagsleft@true}
\newcommand{\reqnomode}{\tagsleft@false}
\makeatother

\usepackage{latexsym}
\usepackage{amsfonts}

\usepackage{tikz}
\usepackage{float}
\usepackage{lmodern}

\def\dd{\hbox{-}}

\usepackage{marvosym}               

\DeclareMathOperator{\tw}{tw}

\DeclareMathOperator{\dist}{dist}

\DeclareMathOperator{\ta}{tree-\alpha}

\newcounter{tbox}

\newcommand{\sta}[1]{\vspace*{0.3cm}\refstepcounter{tbox}\noindent{ \parbox{\textwidth}{(\thetbox) \emph{#1}}}\vspace*{0.3cm}}

\newcommand{\mylongtitle}[1]{%
  \ifodd\value{page}%
    \protect\parbox{0.97\linewidth}{#1}\hfill%
  \else%
    \hfill\protect\parbox{0.97\linewidth}{#1}%
  \fi%
}

\makeatletter
\newcommand{\otherlabel}[2]{\protected@edef\@currentlabel{#2}\label{#1}}
\makeatother

\mathchardef\mh="2D

\title[Tree independence number II.]{Tree independence number \\
II. Three-path-configurations.}

\author{Maria Chudnovsky$^{\dagger}$}
\author{Sepehr Hajebi $^{\mathsection}$}
\author{Daniel Lokshtanov$^{\ddagger}$}
\author{Sophie Spirkl$^{\mathsection \parallel}$}
\thanks{$^{\mathsection}$Department of Combinatorics and Optimization, University of Waterloo, Waterloo, Ontario, Canada}
\thanks{$^{\dagger}$ Princeton University, Princeton, NJ, USA. Supported by  NSF-EPSRC Grant DMS-2120644 and by AFOSR grant FA9550-22-1-0083.}
\thanks{$^{\ddagger}$ Department of Computer Science, University of California Santa Barbara, Santa Barbara, CA, USA. Supported by NSF grant CCF-2008838.}
\thanks{$^{\parallel}$ We acknowledge the support of the Natural Sciences and Engineering Research Council of Canada (NSERC), [funding reference number RGPIN-2020-03912].
Cette recherche a \'et\'e financ\'ee par le Conseil de recherches en sciences naturelles et en g\'enie du Canada (CRSNG), [num\'ero de r\'ef\'erence RGPIN-2020-03912].  This project was funded in part by the Government of Ontario. This research was conducted while Spirkl was an Alfred P. Sloan fellow.}

\date {\today}

\allowdisplaybreaks

\begin{document}

\maketitle
\begin{abstract}
  A \textit{three-path-configuration} is a graph consisting of three
  pairwise internally-disjoint paths the union of every two of which is
  an induced cycle of length at least four. A graph is \textit{3PC-free} if no induced subgraph
  of it is a three-path-configuration.
  We prove that 3PC-free graphs   have poly-logarithmic tree independence
  number. More explicitly, we show that there exists a constant $c$ such that
  every $n$-vertex 3PC-free graph graph
  has a tree decomposition  in which every bag has stability
  number at most $c (\log n)^2$. This implies that  the \textsc{Maximum Weight Independent Set} problem, as well as
 several other natural algorithmic problems,
  that are  known to be 
  \textsf{NP}-hard in general, can be solved in quasi-polynomial time if
  the input graph is 3PC-free.
\end{abstract}

\section{Introduction}
All graphs in this paper are finite and simple and all logarithms are base $2$. We include the following standard definitions for the reader's convenience (see, for example, \cite{tw15}). 
Let $G = (V(G),E(G))$ be a graph. For $X \subseteq V(G)$, we denote by $G[X]$ the subgraph of $G$ induced by $X$, and by $G \setminus X$ the subgraph of $G$ induced by $V(G) \setminus X$. We use induced subgraphs and their vertex sets interchangeably.
For graphs $G,H$ we say that $G$ {\em contains $H$} if $H$ is
isomorphic to $G[X]$ for some $X \subseteq V(G)$; otherwise, we say that
$G$ is {\em $H$-free}. 
For a family $\mathcal{H}$ of graphs, $G$ is said to be
{\em $\mathcal{H}$-free} if $G$ is $H$-free for every $H \in \mathcal{H}$.

Let $v \in V(G)$. The \emph{open neighborhood of $v$}, denoted by $N_G(v)$, is the set of all vertices in $V(G)$ adjacent to $v$. The \emph{closed neighborhood of $v$}, denoted by $N_G[v]$, is $N(v) \cup \{v\}$. Let $X \subseteq V(G)$. The \emph{(open) neighborhood of $X$}, denoted $N_G(X)$, is the set of all vertices in $V(G) \setminus X$ with at least one neighbor in $X$. The \emph{closed neighborhood of $X$}, denoted by $N_G[X]$, is $N_G(X) \cup X$. When there is
no danger of confusion, we often omit the subscript $G$.
Let $Y \subseteq V(G)$ with $X \cap Y = \emptyset$. We say $X$ is \textit{complete} to $Y$ if all possible edges with an end in $X$ and an end in $Y$ are present in $G$, and $X$ is \emph{anticomplete}
to $Y$ if there are no edges between $X$ and $Y$.

For a graph $G = (V(G),E(G))$, a \emph{tree decomposition} $(T, \chi)$ of $G$ consists of a tree $T$ and a map $\chi: V(T) \to 2^{V(G)}$ with the following properties: 
\begin{enumerate}[(i)]
\itemsep -.2em
    \item For every vertex $v \in V(G)$, there exists $t \in V(T)$ such that $v \in \chi(t)$. 
    
    \item For every edge $v_1v_2 \in E(G)$, there exists $t \in V(T)$ such that $v_1, v_2 \in \chi(t)$.
    
    \item For every vertex $v \in V(G)$, the subgraph of $T$ induced by $\{t \in V(T) \mid v \in \chi(t)\}$ is connected.
\end{enumerate}

The \emph{width} of a tree decomposition $(T, \chi)$ is $\max_{t \in V(T)} |\chi(t)|-1$. The \emph{treewidth} of $G$, denoted by $\tw(G)$, is the minimum width of a tree decomposition of $G$. 
Treewidth was first introduced by Robertson and Seymour in their work on  graph minors.
A bound on the treewidth of a graph provides important information  about its structure \cite{RS-GMXVI};
it is also useful from the algorithmic perspective \cite{Bodlaender1988DynamicTreewidth}. As a result
treewidth has been extensively studied in both structural and algorithmic graph theory.

A {\em stable (or independent) set} in a graph $G$ is a set of pairwise non-adjacent vertices of $G$. The {\em stability (or independence) number} $\alpha(G)$ of $G$ is the size of a maximum stable set in $G$. Given a graph $G$ with weights on its
vertices, the \textsc{Maximum Weight Independent Set (MWIS)} problem is
the problem of finding a stable set in $G$ of maximum total weight.
This problem is known to be \textsf{NP}-hard \cite{alphahard}, but it can be solved
efficiently (in polynomial time)  in graphs of  bounded treewidth.
Closer examination of the algorithm motivated 
Dallard, Milani\v{c} and \v{S}torgel \cite{dms2} to define 
a related graph width parameter, specifically targeting the complexity of the
\textsc{MWIS} problem.
The {\em independence number} of a tree decomposition 
$(T, \chi)$ of $G$ is $\max_{t \in V(T)} \alpha(G[\chi(t)])$. The {\em tree independence number} of $G$, denoted $\ta(G)$, is the minimum independence number of a tree decomposition of $G$. Graphs with large treewidth and  small
$\ta$ are graphs whose large  treewidth can be  explained by the presence of
a large clique. 
It is shown in \cite{dms2} that if a graph is given together with a tree decomposition with bounded independence number, then the MWIS problem 
can be solved in polynomial time. Moreover, \cite{dfgkm} presents an algorithm
that constructs  such tree decompositions efficiently in graphs of bounded $\ta$, yielding an efficient algorithm for the MWIS problem for graphs of bounded $\ta$.

We need the following standard definitions (see, for example, \cite{tw3, tw15}). A {\em hole} in a graph is an induced cycle of length at least four.
A {\em path} in a graph is an induced subgraph that is a path. The
{\em length} of a path or a hole is the number of edges in it.
Given a path $P$ with ends $a,b$, the {\em interior} of $P$, denoted by
$P^*$, is the set $P \setminus \{a,b\}$.

A {\em theta} is a graph consisting of two distinct  vertices $a, b$ and three
paths $P_1, P_2, P_3$ from $a$ to $b$, such that $P_i \cup P_j$ is a hole
for every $i,j \in \{1,2,3\}$. It follows that $a$ is non-adjacent to $b$
and the sets $P_1^*,P_2^*, P_3^*$ are pairwise disjoint and
anticomplete to each other.
If a graph $G$ contains an induced subgraph $H$ that is a theta, and $a, b$ are the two vertices of degree three in $H$, then we say that $G$ contains a
theta \emph{with ends $a$ and $b$}.

A {\em pyramid} is a graph consisting of a vertex $a$ and a triangle $\{b_1, b_2, b_3\}$, and three paths $P_i$ from $a$ to $b_i$ for $1 \leq i \leq 3$,
such that  $P_i \cup P_j$ is a hole for every $i,j \in \{1,2,3\}$.
It follows 
that $P_1 \setminus a, P_2 \setminus a , P_3 \setminus a$ are pairwise disjoint,
and the only edges between them are of the form $b_ib_j$.
It also follows that 
at most one of $P_1, P_2, P_3$ has length exactly one. We say that $a$ is the {\em apex} of the pyramid and that $b_1b_2b_3$ is the {\em base} of the pyramid. 

A {\em generalized prism} is a graph consisting of two triangles $\{a_1,a_2,a_3\}$ and  $\{b_1, b_2, b_3\}$, and three paths $P_i$ from $a_i$ to $b_i$
for $1 \leq i \leq 3$, and such that  $P_i \cup P_j$ is a hole for every $i,j \in \{1,2,3\}$. It follows 
that $P_1^*, P_2^*, P_3^*$ are pairwise disjoint and anticomplete to each other,
$|\{a_1,a_2,a_3\} \cap \{b_1,b_2,b_3\}| \leq 1$, and if
$a_1=b_1$, then $P_2^* \neq \emptyset$ and $P_3^* \neq \emptyset$.
Moreover, the only edges between $P_i$ and $P_j$ are 
$a_ia_j$ and $b_ib_j$.
A {\em prism} is a generalized prism whose triangles are disjoint. 
A {\em pinched prism} is a generalized prism whose triangles meet.

A {\em three-path-configuration (3PC)}
is a graph that is either a theta, or
pyramid, or a generalized prism
(see Figure~\ref{fig:3pc}).
It is easy to check that this
definition is equivalent to the one in the abstract.
Let $\mathcal{C}$ be the class of (theta, pyramid, generalized prism)-free graphs; $\mathcal{C}$ is also known as the class of {\em 3PC-free} graphs.

\begin{figure}[t!]
    \centering
    \includegraphics[scale=0.6]{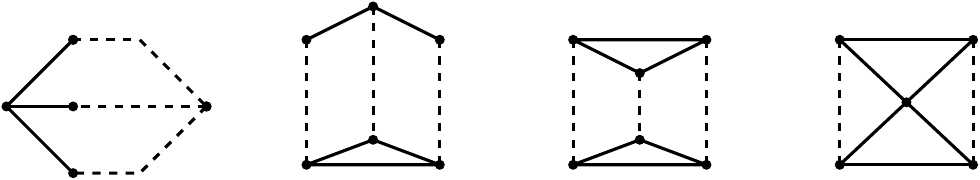}
    \caption{The three-path-configurations. From left to right: A theta, a pyramid, a prism and a pinched prism (dashed lines depict paths of non-zero length).}
    \label{fig:3pc}
\end{figure}

The following is the main result of \cite{tw3}:
\begin{theorem}[Abrishami, Chudnovsky, Hajebi, Spirkl \cite{tw3}]
\label{thetaKt}
For every integer $t>0$ there exists a constant $c(t)$ such that
for every $n$-vertex graph 
$G \in \mathcal{C}$ that  contains no clique of size $t$,
$\tw(G)\leq c(t) \log n$. 
\end{theorem}
This is a strengthening of a conjecture of \cite{mainconj}
that theta-free graphs with no $3$-vertex clique have logarithmic treewidth.
It was also shown in \cite{mainconj} that there exist triangle-free graphs in $\mathcal{C}$ with arbitrarily large treewidth (in fact, treewidth logarithmic in
the number of vertices), and so the bound of Theorem~\ref{thetaKt} is
asymptotically best
possible. A consequence of Theorem~\ref{thetaKt} is that the \textsc{MWIS} problem (as well as many others) can be solved in polynomial time on 3PC-free graphs with bounded
clique number.

It is now natural to ask  about 3PC-free  graphs with no bound on the clique number. Since the complete bipartite graph $K_{2,3}$ is a theta, and therefore
is forbidden in graphs in $\mathcal{C}$, one would expect these graphs to
behave well with respect to $\ta$.
Our main result  here  confirms this. We prove:
\begin{theorem} \label{treealpha}
  There exists a constant $c$ such that for every integer $n>1$ every
  $n$-vertex graph $G \in \mathcal{C}$ has tree independence number 
  at most $c(\log n)^2$.
\end{theorem}

Note that since the class of theta-free graphs is ``$\chi$-bounded'' (see
\cite{chisurvey}  for details), Theorem~\ref{treealpha} yields  a weakening of Theorem~\ref{thetaKt}, that for every integer $t>0$, there exists a constant $c(t)$ such that
for every $n$-vertex graph 
$G \in \mathcal{C}$ that  contains no clique of size $t$,
$\tw(G)\leq c(t) (\log n)^2$. On the other hand, since the  only construction of 3PC-graphs with large treewidth known to date
is the construction of 
\cite{mainconj}
where all graphs have clique number at most four,
we do not know if the bound of Theorem~\ref{treealpha}
is asymptotically tight, or whether it can be made linear in $\log n$  (in which case, it would imply Theorem~\ref{thetaKt}).

Another result in this paper that may be of independent interest is the
following:
\begin{theorem} \label{thm:banana}
   Let $G \in \mathcal{C}$ with $|V(G)|=n$, and let $a,b \in V(G)$ be
  non-adjacent. Then there is a set $X \subseteq V(G) \setminus \{a,b\}$
  with $\alpha(X) \leq 32 \log n$ and such that every component of
  $G \setminus X$ contains at most one of $a,b$.
\end{theorem}

\subsection{Proof outline and organization}
The proof of Theorem~\ref{treealpha} follows an outline similar to \cite{tw15}, but requires several new
techniques and ideas. We sketch it in this subsection, postponing the precise
definitions for later.
We begin by exploring the effect that the
presence of ``useful wheels'' has on 3PC-free graphs, and show that every
useful wheel can be broken by a cutset that is contained in the
 union of the neighborhoods of three vertices.
This is done in Section~\ref{cutsets}.

For a graph $G$ a function  $w: V(G) \rightarrow [0,1]$ is a {\em normal weight function} on $G$ if $w(V(G))=1$.
Let $c \in [0, 1]$ and let $w$ be a normal weight function on $G$. A set $X \subseteq V(G)$ is a {\em $(w,c)$-balanced separator} if
$w(D) \leq  c$ for every component $D$ of $G \setminus X$. The set $X$ is a {\em $w$-balanced separator} if $X$ is a $(w,\frac{1}{2})$-balanced separator.
We show:
\begin{theorem}
There is an integer  $d$ with the following property.
Let $G \in \mathcal {C}$, and let $w$ be a normal weight function on $G$.
Then there exists $Y \subseteq V(G)$ such that
\begin{itemize}
\item $|Y| \leq d$, and
\item $N[Y]$ is a $w$-balanced separator in $G$.
\end{itemize}
\end{theorem}
This is done in  Section~\ref{sec:domsep}; the proof is similar to
the proof of an analogous statement in  \cite{tw15}.

In Section~\ref{banana} we prove Theorem~\ref{thm:banana}.
The key insight here is that  a stronger result can (and should) be
proved, showing that every two
``cooperative subgraphs'', disjoint and anticomplete to each other,
can be separated by removing a set with logarithmic stability number.
The proof of this strengthening follows by relatively standard
structural analysis.

In Section~\ref{sec:smallsep} we develop a technique that
uses  results of
Section~\ref{sec:domsep} and Section~\ref{banana} and produces a
balanced separator of small stability number in a graph. This
technique does not depend on the particular graph-class in question,
but only on the validity of statements similar to Theorems~ \ref{balancedsep}
and \ref{ablogn}.
We also rely on a lemma from
Section~\ref{sec:TheLemma}, which is proved here for theta-free graphs, but
can be generalized in several ways.
Section~\ref{sec:smallsep}
is completely different from \cite{tw15}, and requires several new ideas.

In Section~\ref{sec:theend} we deduce  Theorem~\ref{treealpha} from
the building blocks developed so far.
We finish with Section~\ref{sec:alg} discussing the algorithmic implications of
Theorem~\ref{treealpha}.

\section{Structural results} \label{cutsets}
In this section we prove a theorem  asserting the
existence of certain cutsets in graphs in $\mathcal{C}$.

Let $G$ be a graph. Let $X,Y,Z \subseteq V(G)$. We say that $X$
{\em separates} $Y$ from $Z$ if no component of $G \setminus X$
meets both $Y$ and $Z$.
Let $W$ be a hole in $G$ and $v \in G \setminus W$. A {\em sector} of $(W,v)$ is a path $P$ of $W$ of length at least one, such that both ends of $P$ are  adjacent to $v$, and $v$ is
anticomplete to $P^*$. A sector $P$ is {\em long} if $P^* \neq \emptyset$. 
A {\em useful wheel} in $G$ is a pair $(W,v)$ where $W$ is a hole of length at least seven and $(W,v)$ has at least two long sectors.
We prove:

\begin{theorem}
  \label{thm:wheelseparator}
  Let $G \in \mathcal{C}$ and let $(W,v)$ be a useful  wheel in $G$.
  Let $S$ be a long sector of $W$ with ends $s_1,s_2$. Then
  $((N(s_1) \cup N(s_2)) \setminus W) \cup N(v)$ separates
  $S^*$ from $W \setminus S$.
  \end{theorem}

\begin{proof}
  Let $X=((N(s_1) \cup N(s_2)) \setminus W) \cup N(v)$.
    Suppose for a contradiction that there is a component of $G\setminus X$ intersecting both $S^*$ and $W\setminus S$. It follows that there is a path $P=p_1\dd \dots \dd p_k$ in $G\setminus X$, possibly with $k=1$, such that $p_1$ has a neighbor in $S^*$ and $p_k$ has a neighbor in $W\setminus S$. In particular, $P$ is disjoint from and anticomplete to $\{s_1,s_2,v\}$. 
    
    Choose $P$ with $|P|=k$ as small as possible. It follows that
    \begin{itemize}
        \item we have $P\subseteq G\setminus (W\cup X)$;
        \item $P^*$ is anticomplete to $W\cup \{v\}$;
        \item if $k>1$, then $p_1$ is anticomplete to $W\setminus S^*$ and $p_k$ is anticomplete to $S$.
    \end{itemize}
Let $t_1$ and $t_2$ be the (unique) neighbors of $s_1$ and $s_2$ in $W\setminus S^*$, respectively. Since $(W,v)$ has at least two long sectors, it follows that $s_1,s_2,t_1,t_2$ are all distinct, and that $W\setminus N[S]\neq  \emptyset$. In particular, since $W\cup \{v\}$ is not a pyramid, pinched prism, or theta, it follows that $v$ has a neighbor in $w\in W\setminus N[S]$.

Traversing $S$ from $s_1$ to $s_2$, let $u_1$ and $u_2$ be the first and the last neighbor of $p_1$ in $S$, respectively. It follows that $u_1,u_2\in S^*$. We deduce:

\sta{\label{st:p_knbrs} $N(p_k)\cap (W\setminus S)\subseteq \{t_1,t_2\}$.}

Suppose not. Then there is a path $Q$ in $G$ from $p_k$ to $v$ such that $Q^*\subseteq W\setminus N[S]$. Assume that $u_1=u_2$. Then there is a theta in $G$ with ends $u_1,v$ and paths $u_1\dd S\dd s_1\dd v, u_1\dd S\dd s_2\dd v, u_1\dd p_1\dd P\dd p_k\dd Q\dd v$. Next, assume that $u_1$ and $u_2$ are distinct and non-adjacent. Then there is a theta in $G$ with ends $p_1,v$ and paths $p_1\dd u_1\dd S\dd s_1\dd v, p_1\dd u_2\dd S\dd s_2\dd v, p_1\dd P\dd p_k\dd Q\dd v$. Since $G$ is theta-free, it follows that $u_1,u_2$ are distinct and adjacent. But now there is a pyramid in $G$ with apex $v$, base $p_1u_1u_2$ and paths $p_1\dd P\dd p_k\dd Q\dd v, u_1\dd S\dd s_1\dd v, u_2\dd S\dd s_2\dd v$, a contradiction. This proves \eqref{st:p_knbrs}.

\sta{\label{st:u1=u2} We have $u_1=u_2$.}

Suppose not. By \eqref{st:p_knbrs} and without loss of generality, we may assume that $p_k$ is adjacent to $t_1$.
Assume first that $t_1$ and $v$ are not adjacent. If $u_1$ and $u_2$ are not adjacent either, then there is a theta in $G$ with ends $p_1,s_1$ and paths $p_1\dd u_1\dd S\dd s_1, p_1\dd u_2\dd S\dd s_2\dd v\dd s_1, p_1\dd P\dd p_k\dd t_1\dd s_1$, and if $u_1,u_2$ are adjacent, then there is a pyramid in $G$ with apex $s_1$, base $p_1u_1u_2$ and paths $p_1\dd P\dd p_k\dd t_1\dd s_1, u_1\dd S\dd s_1, u_2\dd S\dd s_2\dd v\dd s_1$. Since $G$ is (theta, pyramid)-free, it follows that $t_1$ and $v$ are adjacent. Assume that $u_1$ and $u_2$ are not adjacent. Then there is a pyramid in $G$ with apex $p_1$, base $s_1t_1v$ and  paths $s_1\dd S\dd u_1\dd p_1, t_1\dd  p_k\dd P\dd p_1, v\dd s_2\dd S\dd u_2\dd p_1$. Again, since $G$ is pyramid-free, it follows that $u_1,u_2$ are adjacent. But now there is a prism in $G$ with triangles $u_1p_1u_2, s_1t_1v$ and paths $u_1\dd S\dd s_1, p_1\dd P\dd p_k\dd t_1, u_2\dd S\dd s_2\dd v$, a contradiction. This proves \eqref{st:u1=u2}.
\medskip

Henceforth, let $u=u_1=u_2$. It follows that:

\sta{\label{st:p1=pk} We have $k=1$.}

Suppose that $k>1$. Since $W \cup \{p_k\}$ is not a theta with ends
$t_1,t_2$, we may assume by
\eqref{st:p_knbrs} and without loss of generality, that $N(p_k)\cap (W\setminus S)= \{t_1\}$. But then $W\cup P$ is a theta in $G$ with ends $u,t_1$, a contradiction. This proves \eqref{st:p1=pk}.

\medskip

Henceforth, let $p=p_1=p_k$.
Since $W \cup \{p\}$ is not a theta with one end $u$ and the other end in
$\{t_1,t_2\}$, \eqref{st:p_knbrs} implies that 
$N(p)\cap W=\{t_1,t_2,u\}$. Recall that $v$ has a neighbor $w\in W\setminus N[S]$.

\sta{\label{st:t1t2} We have $t_1w\in E(G)$ and $t_2v\notin E(G)$. Similarly, we have $t_2w\in E(G)$ and $t_1v\notin E(G)$.}

Suppose not. Then we may assume, without loss of generality, that either $t_1$ and $w$ are not adjacent, or $t_2$ and $v$ are adjacent. In either case, it follows that there is a path $Q$ in $G$ from $p$ to $v$ such that $t_2\in P^*\subseteq W\setminus (S\cup N(t_1))$. Now
there is a theta in $G$ with ends $p,s_1$ and paths $p\dd u\dd S\dd s_1, p\dd t_1\dd s_1, p\dd Q\dd v\dd s_1$, a contradiction.
This proves \eqref{st:t1t2}.
\medskip

We will now finish the proof. From \eqref{st:t1t2}, it follows that $W\setminus S=t_1\dd w\dd t_2$ and $N(v)\cap W=\{s_1,s_2,w\}$. Recall also that $N(p)\cap W=\{t_1,t_2,u\}$. Since $|W|>6$,  it follows either $s_1\dd S\dd u$ or $s_2\dd S\dd u$, say the former, has non-empty interior. But then there is a theta in $G$ with ends $s_1,u$ and paths $s_1\dd S\dd u, s_1\dd t_1\dd p\dd u$ and $s_1\dd v\dd s_2\dd S\dd u$, a contradiction.
\end{proof}

\section{Dominated balanced separators} \label{sec:domsep}

The goal of this section is to prove the following:

\begin{theorem} \label{balancedsep}
There is an integer  $d$ with the following property.
Let $G \in \mathcal {C}$ and let $w$ be a normal weight function on $G$.
Then there exists $Y \subseteq V(G)$ such that
\begin{itemize}
\item $|Y| \leq d$, and
\item $N[Y]$ is a $w$-balanced separator in $G$.
\end{itemize}
\end{theorem}

We follow the outline of the proof of Theorem 8.1 in \cite{tw15}.
First we repeat several definitions from \cite{tw15}.
Let $G$ be a graph, let $P=p_1 \dd \dots \dd p_n$ be a path in $G$ and let
$X=\{x_1, \dots, x_k\}  \subseteq G \setminus P$. We say that $(P,X)$ is an
{\em alignment} if 
$N_P(x_1)=\{p_1\}$, $N_P(x_k)=\{p_n\}$, 
every vertex of $X$  has a neighbor in $P$, and there exist
$1< j_2 < \dots< j_{k-1} < j_k =  n$ such that
$N_P(x_i) \subseteq p_{j_i} \dd P \dd p_{j_{i+1}-1}$ for $i \in \{2, \dots, k-1\}$.
We also say that $x_1, \dots, x_k$ is {\em the order on $X$ given by the
  alignment $(P,X)$}.
An alignment $(P,X)$ is {\em wide}  if 
each of $x_2, \dots, x_{k-1}$ has two non-adjacent neighbors in  $P$,
{\em spiky} if  each of $x_2, \dots, x_{k-1}$ has exactly one neighbor in $P$ and
{\em triangular} if 
each of $x_2, \dots, x_{k-1}$ has exactly two neighbors in $P$  and they are
adjacent. An alignment is {\em consistent} if it is
wide, spiky or triangular.

  The first step in the proof of Theorem~\ref{balancedsep} is the following:
  \begin{theorem} 
\label{thm:twosides_path}
  For every integer $x\geq 1$, there exists an integer $\sigma=\sigma(x) \geq 1$ with the following property.
    Let $G \in \mathcal{C}$ and assume that $V(G)=D_1 \cup D_2 \cup Y$
  where
  \begin{itemize}
  \item  $Y$ is a stable set with $|Y| = \sigma$,
  \item $D_1$ and $D_2$ are components of $G \setminus Y$,
  \item $N(D_1)=N(D_2)=Y$,  
  \item $D_1=d_1 \dd \dots \dd d_k$ is a path, and
  \item for every $y \in Y$ there exists $i(y) \in \{1, \dots, k\}$ such
    that $N(d_{i(y)}) \cap Y=\{y\}$.
  \end{itemize}
    Then there exist $X \subseteq Y$ with $|X|=x+2$ and a subpath $H_1$ of $D_1$ as well as $H_2 \subseteq D_2$ such that 
\begin{enumerate}
    \item  $(H_1,X)$ is a consistent alignment, and for every vertex $x$ in $X$ except for at most two of them, $N_{D_1}(x) = N_{H_1}(x)$.
    \item One of the following holds.
    \begin{itemize}
     \item We have $|H_2|=1$ (so $H_2\cup X$ is a star), and $(H_1,X)$ is wide.
        \item $(H_2,X)$ is a consistent alignment, the orders given on X by $(H_1,X)$ and by $(H_2,X)$ are the same, and at least one of $(H_1,X)$ and $(H_2,X)$ is wide.
    \end{itemize}
\end{enumerate}
\end{theorem}

The proof of Theorem~\ref{thm:twosides_path} requires two preliminary results.
The first one  is Theorem~\ref{thm:connectifier2general} below from
\cite{tw15}. Following \cite{tw15}, by a {\em caterpillar} we mean a tree $C$ with
  maximum degree three such that  there exists  a path $P$ of $C$ where all
  branch vertices of $C$ belong to $P$.
  (Our definition of a caterpillar is
  non-standard for two reasons: a caterpillar is often allowed to be of
  arbitrary maximum degree, and a spine often contains all vertices of degree
  more than one.) 
A {\em claw} is the graph $K_{1,3}$.
  For a graph $H$, a vertex $v$ of $H$ is
  said to be \textit{simplicial} if $N_H(v)$ is a clique.

\begin{theorem} [{Chudnovsky, Gartland, Hajebi, Lokshtanov, Spirkl; Theorem 5.2 in  \cite{tw15}}]
    
\label{thm:connectifier2general}
For every integer $h \geq 1$, there exists an integer $\mu=\mu(h)\geq 1$ with the following property. Let $G$ be a connected graph.
Let $Y \subseteq G$ such that  $|Y|\geq \mu$, $G \setminus Y$ is connected  and every vertex of $Y$ has a neighbor in $G \setminus Y$.
  Then there is a set $Y' \subseteq Y$ with $|Y'|=h$ and an induced subgraph $H$ of $G \setminus Y$ for which one of the following holds.
  \begin{itemize}
  \item $H$ is a path and every vertex of $Y'$ has a neighbor in $H$.
  \item  $H$ is a caterpillar, or the line graph of a caterpillar, or a subdivided star or the line graph of a subdivided star. Moreover, every vertex of
    $Y'$ has a unique neighbor in $H$ and  every vertex of
      $H \cap N(Y')$ is simplicial in $H$.
          \end{itemize}
\end{theorem}

The second one is:

\begin{lemma}\label{lem:intervalgraph}
    Let $c,x\geq 1$ be integers. Let $G$ be a theta-free graph and assume that $V(G)=D_1 \cup D_2 \cup Y$
  where
  \begin{itemize}
  \item $Y$ is a stable set with $|Y| = (3x+2)(c+2)$;
  \item $D_1$ and $D_2$ are components of $G \setminus Y$;
  \item $N(D_1)=N(D_2)=Y$;  
  \item $D_1$ is a path; and
  \item for every $d\in D_1$, we have  $|N(d) \cap Y|\leq c$.
  \end{itemize}
    Then there exist $X \subseteq Y$ with $|X|=x+2$ and a subpath $H_1$ of $D_1$ such that: 
    \begin{itemize}
        \item 
    $(H_1,X)$ is a consistent alignment.
        \item For all but at most two vertices of $X$, all their neighbors in $D_1$ are contained in $H_1$. 
    \end{itemize}
\end{lemma}
\begin{proof}
  For every vertex $y\in Y$, let $P_y$ be the path in $D_1$ such that $y$ is complete to the ends of $P_y$ and anticomplete to $D_1\setminus P_y$. Let $I$ be the graph with $V(I)=Y$, such that two distinct vertices $y,y'\in Y$ are adjacent in $I$ if and only if $P_y\cap P_{y'}\neq \emptyset$. Then $I$ is an interval graph, and so by \cite{Golumbic}, $I$ is perfect. Since $|V(I)|=(3x+2)(c+2)$,
  we deduce that $I$ contains either a clique of cardinality $c+2$ of a stable set of cardinality $3x+2$. 
    
  Assume that $I$ contains a clique of cardinality $c+2$. Then there exists $C\subseteq Y$ with $|C|=c+2$ and $d\in D_1$ such that $d \in P_y$ for every $y \in C$. It follows that for every $y\in C$, either $y$ is adjacent to $d$, or $D_1\setminus d$ has two components and $y$ has a neighbor in each of them.
  Since $|N(d) \cap Y| \leq c$,
  we deduce that there are two vertices $y,y'\in C\subseteq Y$ as well as two paths $P_1$ and $P_2$ from $y$ to $y'$ with disjoint and anticomplete interiors contained in $D_1$. On the other hand, since $D_2$ is connected and $N(D_2)=Y$, it follows that there is a path $P_3$ in $G$ from $y$ to $y'$ whose interior is contained in $D_2$. But now there is a theta in $G$ with ends $y,y'$ and paths $P_1, P_2, P_3$, a contradiction.

    We deduce that $I$ contains a stable set $S$ of cardinality $3x+2$. From the definition of $I$, it follows that there is a subpath $H_1$ of $D_1$ such that $(H_1, S)$ is an alignment. Since every vertex of $S$, except for the fist and last,  either has two non-adjacent neighbors in $H_1$, or has exactly one neighbor in $H_1$, or has exactly two neighbors in $H_1$ and
  they are adjacent,  there exists $X\subseteq S\subseteq Y$ with $|X|=x+2$ such that $(H_1,X)$ is a consistent alignment. This completes the proof of Lemma~\ref{lem:intervalgraph}.
\end{proof}

We are now ready to prove Theorem~\ref{thm:twosides_path}:

\begin{proof}[Proof of Theorem~\ref{thm:twosides_path}]
    Let $\sigma(x)=18\mu(3((x+2)^2+1)(x+1))$, where $\mu(\cdot)$ comes from Theorem~\ref{thm:connectifier2general}. We begin with the following:

    \sta{\label{st:fewnbrs} Every vertex in $D_1$ has at most four neighbors in $Y$.}

   Suppose for a contradiction that for some $i\in \{1,\dots, k\}$, there is a subset $Z\subseteq Y$ of cardinality five such that $d_i$ is complete to $Z$. It follows that for every $y\in Z$, we have $i(y)\neq i$, and so there is $3$-subset $T$ of $Z$ such that either $i(y)<i$ for all $y\in T$ or $i<i(y)$ for all $y\in T$. Consequently, there are two distinct vertices $y,y'\in T\subseteq Z\subseteq Y$ for which $d_i$ is disjoint from and anticomplete to $d_{i(y)}\dd D_1\dd d_{i(y')}$. On the other hand, since $D_2$ is connected and $N(D_2)=Y$, it follows that there is a path $Q$ in $G$ from $y$ to $y'$ whose interior is contained in $D_2$. But now there is a theta in $G$ with ends $y,y'$ and paths $y\dd d_{i(y)}\dd D_1\dd d_{i(y')}\dd y',y\dd d_i\dd y', Q$, a contradiction. This proves \eqref{st:fewnbrs}.

   \medskip

   From \eqref{st:fewnbrs}, Lemma~\ref{lem:intervalgraph} and the choice of $\sigma(x)$, it follows that:

    \sta{\label{st:alignmentonD_1} There exists $Y_1\subseteq Y$ with $|Y_1|=\mu(3((x+2)^2+1)(x+1))$ and a subpath $H_1$ of $D_1$ such that $(H_1,Y_1)$ is consistent alignment.}

    Henceforth, let $Y_1$ be as in \eqref{st:alignmentonD_1}. Since $G_1=G[Y_1\cup D_2]$ and $G_1\setminus Y_1=D_2$ are both connected, we can apply Theorem~\ref{thm:connectifier2general} to $G_1$ and $Y_1$. It follows that there is a set $Y' \subseteq Y_1$ with $|Y'|=3((x+2)^2+1)(x+1)$ and an induced subgraph $H$ of $G_1 \setminus Y_1=D_2$ for which one of the following holds.
  
    \begin{itemize}
    \item $H$ is a path and every vertex of $Y'$ has a neighbor in $H$.
    \item  $H$ is a caterpillar, or the line graph of a caterpillar, or a subdivided star or the line graph of a subdivided star. Moreover, every vertex of
    $Y'$ has a unique neighbor in $H$ and  every vertex of
      $H \cap N(Y')$ is simplicial in $H$.
      \end{itemize}

    Assume that the second bullet above holds. By \eqref{st:alignmentonD_1}, $(H_1,Y')$ is a consistent alignment. We will show that we get a theta or a prism or a pyramid. Assume first that $(D_1, Y')$ is spiky.  If $H$ is a caterpillar or a subdivided star,  we get a theta with one end in $D_1$ and the other one in $D_2$; and 
    if $H$ is the  line graph of a caterpillar, we get a pyramid with base in
    $H$ and apex in $D_1$.
    Next assume that  $(D_1, Y')$ is triangular.  If $H$ is a caterpillar or a subdivided star,  we get a pyramid  with apex in  $D_2$ and base consisitng of
    two vertices of $D_1$ and  a vertex of $Y'$; and
    if $H$ is the  line graph of a caterpillar, we get a prism
    one of whose triangles in in $H$, and the other consists of two vertices of $D_1$ and a vertex of $Y'$.
    Finally assume that $(D_1, Y')$ is wide.
    If $H$ is a caterpillar or a subdivided star,  we get a theta
    with one end in $Y'$ and the other end in $D_2$, and if 
    if $H$ is the  line graph of a caterpillar, we get a pyramid with
    apex in $Y'$ and base in $H$.
    Thus in all cases we get a contradiction.
    It follows that $H$ is indeed a path and every vertex of $Y'$ has a neighbor in $H$.

Now, assume that some vertex in $z\in H$ has at least $x$ neighbors in $Y'$. Choose $X\subseteq N(z)\cap Y'\subseteq Y$ with $|X|=x$. Let $H_2=\{z\}$. By \eqref{st:alignmentonD_1}, $(H_1,X)$ is a consistent alignment. Note that if $(H_1,X)$ is spiky, then $H_1\cup X\cup \{z\}$ contains a theta, and if $(H_1,X)$ is triangular, then $H_1\cup X\cup \{z\}$ contains a pyramid. Therefore, $(H_1,X)$ is wide. But now $X$ and $H_2$ satisfy Theorem~\ref{thm:twosides_path}.

Therefore, we may assume that every vertex in $H$ has fewer than $x$ neighbors in $Y'$. Let $H_2=H$. Since $|Y'|=3((x+2)^2+1)(x+1)$, it follows from Lemma~\ref{lem:intervalgraph} that there exists $X'\subseteq Y'$ with $|X'|=(x+2)^2$ such that $(H_2, X')$ is a consistent alignment. Also, by \eqref{st:alignmentonD_1}, $(D_1,X')$ is a consistent alignment. This, along with the Erd\H{o}s-Szekeres Theorem
\cite{ES}, implies that there exists $X\subseteq X'\subseteq Y'\subseteq Y$ with $|X|=x+2$ such that both $(H_1,X)$ and $(H_2,X)$ are consistent alignments, and the orders given on $X$ by $(H_1,X)$ and $(H_2,X)$ are the same. Moreover, since $G$ is (theta, pyramid, pinched prism)-free, it follows that at least one of $(H_1,X)$ and $(H_2,X)$ is wide. Hence, $X$ and $H_2$ satisfy Theorem~\ref{thm:twosides_path}. This completes the proof.
\end{proof}

Now, as in \cite{tw15}, we will show that the class $\mathcal{C}$ is
``amiable'' and ``amicable'', and then use Theorem~8.6 of \cite{tw15}
to complete to the proof. The details are below.
In \cite{tw15}, a graph class $\mathcal{G}$ is said to be \textit{amiable} if, under the same assumptions as that of Theorem~\ref{thm:twosides_path} for a graph $G\in \mathcal{G}$, there exists $X\subseteq Y$ with $|X|=x+2$, $H_1 \subseteq D_1$ and $H_2\subseteq D_2$ satisfying one of several possible outcomes. In particular, the outcome of Theorem~\ref{thm:twosides_path} is one of the possible outcomes in the definition of an amiable class. Therefore, from Theorem~\ref{thm:twosides_path}, we deduce that:
\begin{corollary}\label{cor:twopath}
  The class $\mathcal{C}$ is amiable.
\end{corollary}

Following \cite{tw15}, for an integer $m>0$, a graph class $\mathcal{G}$ is said to be \textit{$m$-amicable} if $\mathcal{G}$ is amiable, and the following holds for every graph $G \in \mathcal{G}$.  Let $\sigma$ be as in the definition of an amiable class (and so as in Theorem~\ref{thm:twosides_path}) for $\mathcal{G}$ and let $V(G)=D_1\cup D_2\cup Y$ such that $D_1=d_1 \dd \dots \dd d_k, D_2$ and $Y$ satisfy the assumptions of Theorem \ref{thm:twosides_path} with $|Y| = \sigma(7)$. Let $X\subseteq Y$, $H_1\subseteq D_1$ and $H_2\subseteq D_2$ be as  in Theorem \ref{thm:twosides_path} with $|X|=9$, and let $\{x_1, \dots, x_7\} \subseteq X$ such that: 
\begin{itemize}
    \item $x_1, \dots, x_7$ is the order given on $\{x_1, \dots, x_7\}$ by $(H_1,X)$; and 
    \item For every $x \in \{x_1, \dots, x_7\}$, we have $N_{D_1}(x) = N_{H_1}(x)$.
\end{itemize} Let $i$ be maximum such that $x_1$ is adjacent to $d_i$, and let $j$ be minimum such that $x_7$ is adjacent to $d_j$. Then there exists a subset $Z \subseteq D_2 \cup \{d_{i+2}, \dots, d_{j-2}\} \cup  \{x_4\}$
    with $|Z| \leq m$  such that $N[Z]$ separates  $d_i$ from $d_j$.  Consequently, $N[Z]$ separates $d_1 \dd D_1 \dd d_i$ from
    $d_j \dd D_1 \dd d_k$. 
    We prove:
    
 \begin{theorem}
    \label{thm:separatepath}
    The class $\mathcal{C}$ is $3$-amicable.
 \end{theorem}
\begin{proof}
By Corollary~\ref{cor:twopath}, $\mathcal{C}$ is an amiable class. With same notation as in the definition of a $3$-amicable class, our goal is to show that there exists a subset
    $Z \subseteq D_2 \cup \{d_{i+2}, \dots, d_{j-2}\} \cup  \{x_4\}$
    with $|Z| \leq 3$  such that $N[Z]$ separates  $d_i$ from $d_j$.
    Consequently, $N[Z]$ separates $d_1 \dd D_1 \dd d_i$ from
    $d_j \dd D_1 \dd d_k$.

    Let $l\in \{1,\ldots, k\}$ be minimum such that $x_4$ is adjacent to $d_l$, and let $m\in \{1,\ldots, k\}$ be maximum such that $x_4$ is adjacent to $d_m$. It follows that $i+2<l\leq m< j-2$.
    
    Let $R$ be the (unique) path in $H_2$ with ends $r_1,r_2$ (possibly $r_1=r_2$) such that $x_1$ is adjacent to $r_1$ and anticomplete to $R\setminus \{r_1\}$ and $x_7$ is adjacent to $r_2$ and anticomplete to $R\setminus \{r_2\}$. Then $x_4$ has a neighbor in $R$. Traversing $R$ from $r_1 $ to $r_2$, let $z_1$ and $z_2$ be the first and the last neighbor of $x_4$ in $R$.
    
    Let $W=d_i\dd D_1\dd d_j\dd x_7\dd r_2\dd R\dd r_1\dd x_1\dd d_i$. Then $W$ is a hole in $G$  and $|W| \geq 7$. It follows that $(W,x_4)$ is a useful wheel in $G$. In particular, $S=d_l\dd D_1\dd d_i\dd x_1\dd r_1\dd R\dd z_1$ and $S'=d_m\dd D_1\dd d_j\dd x_7\dd r_2\dd R\dd z_2$ are two long sectors of $(W,x_4)$.
    
    Note that $d_l$ and $z_1$ are the ends of the sector $S$ from $(W,x_4)$. Let $Z=\{x_4,d_l,z_1\}$. Then we have $Z\subseteq D_2 \cup \{d_{i+2}, \dots, d_{j-2}\} \cup  \{x_4\}$,
    $d_i\in S^*\setminus N[Z]$ and $d_j\in W\setminus (S\cup N[Z])$. Hence, by Theorem~\ref{thm:wheelseparator}, $N[Z]$ separates $d_i$ from $d_j$, as desired.
\end{proof}
The following is a restatement of Theorem~8.6 of \cite{tw15}:

\begin{theorem}[Chudnovsky, Gartland, Hajebi, Lokshtanov, Spirkl \cite{tw15}]\label{balancedseptw15}
For every integer $m>0$ and every $m$-amicable graph class $\mathcal{G}$, there is an integer $d>0$ with the following property. Let $\mathcal{G}$ be a graph class which is $m$-amicable. Let $G \in \mathcal {C}$ and let $w$ be a normal weight function on $G$.
Then there exists $Y \subseteq V(G)$ such that
\begin{itemize}
\item $|Y| \leq d$, and
\item $N[Y]$ is a $w$-balanced separator in $G$.
\end{itemize}
\end{theorem}
Now Theorem~\ref{balancedsep} is immediate from Theorems~\ref{thm:separatepath} and \ref{balancedseptw15}.

\section{Separating a pair of vertices} \label{banana}

The goal of this section is to prove the following:

\begin{theorem}
  \label{ablogn}
  Let $G \in \mathcal{C}$ with $|V(G)|=n$, and let $a,b \in V(G)$ be
  non-adjacent. Then there is a set $X \subseteq V(G) \setminus \{a,b\}$
  with $\alpha(X) \leq 16 \times 2 \log n$ and such that every component of
  $G \setminus X$ contains at most one of $a,b$.
  \end{theorem}

We need the following result from \cite{prismfree}.

\begin{lemma}[Abrishami, Chudnovsky,              Dibek, Vu\v{s}kovi\'c \cite{prismfree}]\label{minimalconnected}
Let $x_1, x_2, x_3$ be three distinct vertices of a graph $G$. Assume that $H$ is a connected induced subgraph of $G \setminus \{x_1, x_2, x_3\}$ such that $V(H)$ contains at least one neighbor of each of $x_1$, $x_2$, $x_3$, and that $V(H)$ is minimal subject to inclusion. Then, one of the following holds:
\begin{enumerate}[(i)]
\item For some distinct $i,j,k \in  \{1,2,3\}$, there exists $P$ that is either a path from $x_i$ to $x_j$ or a hole containing the edge $x_ix_j$ such that
\begin{itemize}
\item $V(H) = V(P) \setminus \{x_i,x_j\}$; and
\item either $x_k$ has two non-adjacent neighbors in $H$ or $x_k$ has exactly two neighbors in $H$ and its neighbors in $H$ are adjacent.
\end{itemize}

\item There exists a vertex $a \in V(H)$ and three paths $P_1, P_2, P_3$, where $P_i$ is from $a$ to $x_i$, such that 
\begin{itemize}
\item $V(H) = (V(P_1) \cup V(P_2) \cup V(P_3)) \setminus \{x_1, x_2, x_3\}$;  
\item the sets $V(P_1) \setminus \{a\}$, $V(P_2) \setminus \{a\}$ and $V(P_3) \setminus \{a\}$ are pairwise disjoint; and
\item for distinct $i,j \in \{1,2,3\}$, there are no edges between $V(P_i) \setminus \{a\}$ and $V(P_j) \setminus \{a\}$, except possibly $x_ix_j$.
\end{itemize}

\item There exists a triangle $a_1a_2a_3$ in $H$ and three paths $P_1, P_2, P_3$, where $P_i$ is from $a_i$ to $x_i$, such that
\begin{itemize}
\item $V(H) = (V(P_1) \cup V(P_2) \cup V(P_3)) \setminus \{x_1, x_2, x_3\} $; 
\item the sets $V(P_1)$, $V(P_2)$ and $V(P_3)$ are pairwise disjoint; and
\item for distinct $i,j \in \{1,2,3\}$, there are no edges between $V(P_i)$ and $V(P_j)$, except $a_ia_j$ and possibly $x_ix_j$.
\end{itemize}
\end{enumerate}
\end{lemma}

For a graph $G$ and  two  subsets $X,Y \subseteq V(G)$
we define {\em the distance in $G$ between $X$ and $Y$} as the length (number of edges) of the shortest
path of $G$ with one end in $X$ and the other in $Y$. We denote the distance
between $X$ and $Y$ by $\dist_G(X,Y)$. Thus $X$ and $Y$ are disjoint
if and only if $\dist_G(X,Y)>0$, and 
$X,Y$ are anticomplete to each other if and only if $\dist_G(X,Y) >1$. 
In order to prove Theorem~\ref{ablogn} we will prove a stronger statement.
Let $H \subseteq G$. We denote by $\delta_G(H)$ the set of vertices of $H$ that
have a neighbor in $G \setminus H$ (so $\delta(H)=N(G \setminus H)$). We say
that $H$ is {\em cooperative}
if one of the following holds:
\begin{itemize}
\item $H$ is a clique, or
\item   $N_H(H \setminus \delta(H))=\delta(H)$ and
  $H \setminus \delta(H)$ is connected.
\end{itemize}

The following lemma summarizes the property of cooperative subgraphs that is of interest to us.

\begin{lemma}
  \label{coopH}
  Let $H \subseteq G$ be cooperative and let 
  $\{n_1,n_2,n_3\}$ be a stable set in $N(H)$.
  Assume that there exist distinct $h_1,h_2,h_3 \in \delta(H)$ such that
  $n_ih_j$ is an edge if and only if $i=j$.
  Then there is $K \subseteq H \cup \{n_1,n_2,n_3\}$
  such that
  \begin{enumerate}
  \item $K$ is a subdivided claw or the line graph of a subdivided claw
  \item $\{n_1,n_2,n_3\}$ is the set of simplicial vertices of $K$.
  \end{enumerate}
\end{lemma}

\begin{proof}

  If $\{h_1,h_2,h_3\}$ is a triangle, then
  $\{n_1,n_2,n_3, h_1,h_2,h_3\}$ is the line graph of a
  subdivided claw and the lemma holds. This we may assume that
  at least one pair $h_ih_j$ is non-adjacent, and in particular
  $H$ is not a clique.
  It follows that  $N_H(H \setminus \delta(H))=\delta(H)$ and
  $H \setminus \delta(H)$ is connected.
  
  Next suppose that $h_1h_2$ and $h_2h_3$ are edges.
  Then $h_1h_3$ is not an edge, and
  $\{n_1,n_2,n_3, h_1,h_2,h_3\}$ is  a
  subdivided claw and the lemma holds.
  Thus we may assume that at most one of the pairs $h_ih_j$
  is an edge.

  Suppose $h_1h_3$ is an edge. Let $P=p_1 \dd \dots \dd p_k$ be path
such that $p_1=h_2$, $p_k$ has a neighbor in $\{h_1,h_3\}$, 
   and  $P \setminus p_1  \subseteq H \setminus \delta(H)$;
choose $P$ with $k$  as small as possible.
If $p_k$ is adjacent to exactly one of $h_1,h_3$,
then $P \cup \{h_1,h_2,h_3,n_1,n_2,n_3\}$ is a subdivided claw;
and if $p_k$ is adjacent to both $h_1$ and $h_3$,
then $P \cup \{h_1,h_2,h_3,n_1,n_2,n_3\}$ is the line graph of
a subdivided claw; in both cases the lemma holds.
Thus we may assume that $\{h_1,h_2,h_3\}$ is a stable set.

Let $R$ be a minimal connected subgraph of $H \setminus \delta(H)$ such that each of $h_1,h_2,h_3$
has a neighbor in $R$.  We apply Lemma~\ref{minimalconnected}. Suppose that the  first outcome holds; we may assume that $R$ is path from $h_1$ to $h_2$.
If $h_3$ has two non-adjacent neighbors in $R$, then 
$R \cup \{h_1,h_2,h_3,n_1,n_2,n_3\}$ contains a  subdivided claw; and
if $h_3$ has exactly two neighbors in $R$ and they are adjacent, then
$R  \cup \{h_1,h_2,h_3,n_1,n_2,n_3\}$ is the line graph of a subdivided claw;  in both cases the theorem holds. If the second outcome of
Lemma~\ref{minimalconnected} holds, then $R  \cup \{h_1,h_2,h_3,n_1,n_2,n_3\}$ is a subdivided claw; and if the third outcome of
Lemma~\ref{minimalconnected} holds, then $R  \cup \{h_1,h_2,h_3,n_1,n_2,n_3\}$ is the line graph of a subdivided claw. Thus in all cases the lemma holds.
\end{proof}

We also need the following:
\begin{lemma}\label{commonnbrs}
  Let $G \in \mathcal{C}$ and let $H_1,H_2$
  be cooperative subgraphs of $G$, disjoint and anticomplete to each other.
  Then $\alpha(N(H_1) \cap N(H_2)) < 17$.
  \end{lemma}

\begin{proof}
  Suppose there is a stable set $N \subseteq N(H_1) \cap N(H_2)$
    with $|N|=17$.
  Suppose first that some vertex $h_1 \in H_1$ has at least five
  neighbors in $N$; let $n_1, \dots, n_5 \in N \cap N(h_1)$.
  If some $h_2 \in H_2$ has three three neighbors in
  $\{n_1, \dots, n_5\}$, say $\{n_1,n_2,n_3\}$,
  then $\{h_1,h_2,n_1,n_2,n_3\}$ is a theta with ends $h_1,h_2$,
  a contradiction. So no such $h_2$ exists. It follows that
  there exist $h_1',h_2',h_3' \in H_2$ and $n_1',n_2', n_3' \in \{n_1, \dots, n_5\}$
  such that $h_i'n_j'$ is an edge if and only if $i=j$. 
  By Lemma~\ref{coopH} there exists $K \subseteq H_2 \cup \{n_1',n_2',n_3'\}$
  such that
  $K$ is a subdivided claw or the line graph of a subdivided claw, and
  $\{n_1',n_2',n_3'\}$ is the set of simplicial vertices of $K$.
    But now $K \cup h_1$ is a theta or a pyramid in $G$, a contradiction.

  It follows that for every $h_1 \in H_1$, $|N(h_1) \cap N| \leq 4$.
  Since $|N|=17$ and $N \subseteq N(H_1)$, there exist
  $h_1, \dots, h_5 \in H_1$, and $n_1, \dots, n_5 \in N$
  such that $h_in_j$ is an edge if and only if $i=j$.

  By renumbering $n_1, \dots, n_5$ if necessary we may assume that
  one of the following holds:
  \begin{itemize}
  \item there exists $k \in H_2$ such that $n_1,n_2,n_3 \in N(k)$;
    in this case set $K_2=\{k,n_1,n_2,n_3\}$, or
  \item  there exist $k_1,k_2,k_3 \in H_2$ such that
    $k_in_j$ is an edge if and only if $i=j$. In this case
    let     $K_2 \subseteq H_2 \cup \{n_1,n_2,n_3\}$ be
  such that
  $K_2$ is a subdivided claw or the line graph of a subdivided claw
  and $\{n_1,n_2,n_3\}$ is the set of simplicial vertices of $K_2$
  (such $K_2$ exists by Lemma~\ref{coopH}).
  \end{itemize}
By Lemma~\ref{coopH} there exists $K_1 \subseteq H_1 \cup \{n_1,n_2,n_3\}$
  such that
  $K_1$ is a subdivided claw or the line graph of a subdivided claw, and
  $\{n_1,n_2,n_3\}$ is the set of simplicial vertices of $K_1$.
  But now $K_1 \cup K_2$ is a theta, a pyramid or a prism in $G$, a contradiction.
  \end{proof}
  
For $X \subseteq V(G)$, a component
$D$ of $G \setminus X$ is {\em full for $X$} if $N(D)=X$.
$X \subseteq V(G)$ is a {\em minimal separator} in $G$ if
there exist two distinct full components for $X$. We will now prove the following strengthening of Theorem~\ref{ablogn}: 

\begin{theorem}
  \label{H1H2logn}
  Let $G \in \mathcal{C}$ with $|V(G)| =n$, and let $H_1,H_2$
  be cooperative subgraphs of $G$, disjoint and anticomplete to each other.
    Then there is a set $X \subseteq V(G) \setminus (H_1 \cup H_2)$
    with $\alpha(X) \leq 16 \times 2 \log (n+1-|H_1|-|H_2|)$ and such that
    $X$ separates $H_1$ from $H_2$.
      \end{theorem}

\begin{proof}
  Write $G_1=G$, $H_2^1=H_2$  and  $N_1=N_{G_1}(H_1)\cap N_{G_1}(H_2)$.
  Define $G_{2}=G_1 \setminus N_1$, 
  $H_2^{2}=N_{G_{2}}[H_2^1]$ and   $N_{2}=N_{G_{2}}(H_1) \cap N_{G_{2}}(H_2^{2})$.
  Let $G_3= G_2  \setminus N_2$.
    \\
  \\
  \sta{$\dist_{G_3}(H_1,H_2)>3$. \label{distance}}

  Let $P=p_1 \dd \dots \dd p_k$ be a shortest path in $G_3$
  from $H_1$ to $H_2$. Then $P$ is a path in $G$, $p_1 \in H_1$,
  $p_k \in H_2$, $P \setminus p_1$ is anticomplete to $H_1$,
  and $P \setminus p_k$ is anticomplete to $H_2$.
  Since $H_1$ is anticomplete to $H_2$, it follows that
  $k \geq 3$.    If $k=3$, then $p_2 \in N_1$, a contradiction.
  Suppose  $k=4$. Then $p_2,p_3 \not \in N_1$. It follows that
  $p_3 \in N_{G_2}(H_2)=H_2^2$, and therefore $p_2 \in N_2$;
  again a contradiction. This proves that $k>4$, and \eqref{distance}
  follows.
  \\
  \\
    It follows immediately from the definition of a cooperative subgraph that:
  \\
  \\
\sta{For $i \in \{1,2\}$, $H_1$ and $H_2^i$ are both  cooperative in $G_i$. \label{icoop}}

Now Lemma~\ref{coopH} implies:
\\
\\
\sta{$\alpha(N_i) <  17$ for every $i \in \{1,2\}$. \label{17}}

If $H_1$ and $H_2$ belong to different components
of $G_3$, then  $N_1 \cup N_2$ 
separates $H_1$ from $H_2$ in $G$. Since
by \eqref{17}, $\alpha(N_1 \cup N_2) \leq 16 \times 2$, the theorem holds.
Thus we may assume that there is a component $F$ of $G_3$
such that $H_1 \cup H_2 \subseteq F$.
\\
\\
\sta{There is  a minimal separator $Z$ in $F$ such that
$\dist_F(Z, H_i) \geq 2$ for $i \in \{1,2\}$,
and there exist distinct  full components $F_1,F_2$ for $Z$ such that
$H_i \subseteq F_i$. \label{distantsep}}

Let $X=N_{F}^2(H_1)$.
It follows that
$X$ separates $H_1$ from $H_2$ in $F$, and since $\dist_F(H_1,H_2) \geq 4$,  we have
$\dist_F(H_2,X) \geq 2$. For $i \in \{1,2\}$, let
$D_i$ be the component of $F \setminus X$ such that $H_i \subseteq D_i$.
Let $D_2'$ be the component of $F \setminus X$
such that $D_2 \subseteq D_2'$, and let $Z=N(D_2)$. Then
$Z \subseteq X$. Let $D_1'$ be the full component of $Z$ such that
$D_1 \subseteq Z$. Now $D_1',D_2'$ are full components for $Z$.
Setting $F_1=D_1'$ and $F_2=D_2'$, \eqref{distantsep} follows.
\\
\\
Let $Z, F_1, F_2$ be as in \eqref{distantsep}.
We are now ready to complete the proof of the theorem.
 The proof is by induction on $n-|H_1|-|H_2|$.
  If $n-|H_1|-|H_2|=0$, then $X=\emptyset$ works. Likewise, if $Z = \emptyset$, we set $X = \emptyset$. 
Since $F_1 \cap F_2 = \emptyset$, we may assume that
$|F_1 \setminus H_1| < \frac{n-|H_1|-|H_2|}{2}$.
Let $F'=F_1 \cup F_2 \cup Z$.
Let $H_2'=F_2 \cup Z$.
Then $\delta_{F'}(H_2')=Z$ and $H_2'$ is cooperative in $F'$.
Since $\dist_{F'}(H_1,Z)\geq 2$, we have that $H_1$ is anticomplete to
$H_2'$. Also
$$|F'| - |H_1|-|H_2'| \leq  |F_1|-|H_1|\le \frac{n-|H_1|-|H_2|-1}{2}.$$ 
Inductively, there exists
$X' \subseteq F' \setminus (H_1 \cup H_2')$
with
$$\alpha(X') \leq 16 \times 2 \log \left(\frac{n-|H_1|-|H_2|-1}{2}+1\right) \leq
16 \times 2 \log \left(n-|H_1|-|H_2|+1\right)-16 \times 2$$
such that $X'$ separates $H_1$ from $H_2' = F_2 \cup Z$ in $F'$. 
Let $X=X' \cup N_1 \cup N_2$.
Then $X$ separates $H_1$ from $H_2$ in $G$.
By \eqref{17}, $\alpha(X) \leq 16 \times 2 \log (n-|H_1|-|H_2|+1)$
as required.
\end{proof}

\section{Large stable subsets in neighborhoods} \label{sec:TheLemma}

In this section, we prove a statement about theta-free graphs which we expect
to use in future papers.

\begin{lemma}\label{TheLemma}
  Let $G$ be a theta-free graph.
Let $c \geq 2$  be integer, and let $Y$ be a set with
$\alpha(Y)>24c^2$.
Let $Z$ be the set of all vertices such that
$\alpha(N(z) \cap Y) \geq \frac{\alpha(Y)}{c}$.
Then $\alpha(Z) < 2c$.
\end{lemma}

\begin{proof}

  Suppose not, and let  $I \subseteq Z$ be a stable set of size $2c$.
  For every $z \in I$, let $J'(z)$ be a stable set in $N(z) \cap Y$ with
$|J'(z)|=\left\lceil \frac{\alpha(Y)}{c} \right\rceil$. 
\\
\\
\sta{For all distinct  $z,z' \in I$,   $\alpha (N[z] \cap N[z']) \leq 2$.
  \label{disjoint}}

Suppose that $\alpha (N[z] \cap N[z']) \geq 3$.
Since $z$ is non-adjacent to $z'$, there exists a stable set of size three
in $N(z) \cap N(z')$, and we get a theta with ends $z,z'$, a contradiction.
This proves~\eqref{disjoint}.
\\
\\
\sta{For all distinct $z,z' \in I$, $|J'(z) \cap N(J'(z')  \setminus N(z))|  \leq 4 $.
  \label{anticomplete}}

Suppose not, and let $\{n_1,..,n_5\} \subseteq J'(z) \cap N(J'(z') \setminus N(z))$.
Then $\{n_1, \dots, n_5\}$ is   a stable set.
If some $h \in  J'(z') \setminus N(z)$ has three neighbors in $\{n_1,..,n_5\}$, then
we get a theta with
ends $z,h$; so no such $h$ exists. It follows that there exist
$h_1,h_2,h_3 \in  J'(z') \setminus N(z)$ such that (permuting $n_1, \dots, n_5$ if necessary)
$n_ih_j$  is an edge if and only if  $i=j$. But now
$\{z,n_1,n_2,n_3,h_1,h_2,h_3,z'\}$ is a theta with ends $z,z'$, again a contradiction.
This proves~\eqref{anticomplete}.
\\
\\
 Let
$J(z)=J'(z)\setminus \bigcup_{z' \in I \setminus  \{z\}} (N[z'] \cup (N(J'(z') \setminus N(z))))$.
By \eqref{disjoint} and \eqref{anticomplete} it follows that 
$$|J(z)| \geq |J'(z) |-6|I| \geq \frac{\alpha(Y)}{c}-12c.$$
But for all distinct $z,z' \in I$ the sets $J(z),J(z')$ are disjoint and anticomplete to
each other; it follows that  $\bigcup_{z \in I} J(z)$ is a stable set of size
  $2c\left(\frac{\alpha(Y)}{c}- 12 c\right)$.
    Consequently,
    $$ 2c\left(\frac{\alpha(Y)}{c}- 12 c\right)        \leq \alpha(Y)$$
and so 
$\alpha(Y) \leq 24c^2$,
a contradiction.
\end{proof}

   \section{From domination to stability}
\label{sec:smallsep}
The last  step in the proof of Theorem~\ref{treealpha} is to transform
balanced separators with small domination number into balanced separators
with small stability number.

The results in this section are more general than what we need in this paper; again they are to be used in future papers in the series.
Let $L,d,r$ be integers. We say that an $n$-vertex graph $G$ is {\em $(L,d,r)$-breakable}
  if
  \begin{enumerate}
    \item
      for every two disjoint and anticomplete cliques 
      $H_1,H_2$ of $G$ with $|H_1| \leq r$ and $|H_2| \leq r$,  there is a set
  $X \subseteq G \setminus (H_1 \cup H_2)$ with $\alpha(X) \leq L$ separating $H_1$ from $H_2$, and
  \item  for every normal weight function $w$ on $G$ and for every induced subgraph
    $G'$ of $G$ there  exists a set $Y \subseteq V(G')$ with $|Y| \leq d$
    such that for every component $D$ of $G'  \setminus N[Y]$,
    $w(D) \leq \frac{1}{2}$.
\end{enumerate}

We prove:
  \begin{theorem}
  \label{smallsep}
  Let $d>0$ be  an integer and let  $C(d)=100 d^2$. 
  Let $L,d,n, r >0$ be integers such that $r \leq d(2+\log n)$ and let $G$ be an
  $n$-vertex $(L,d,r)$-breakable
  theta-free graph.
  Then there exists a $w$-balanced separator $Y$ in $G$ such that
  $\alpha(Y) \leq C(d)  \left\lceil \frac{d (2+\log n)} {r} \right\rceil (2+\log n) L$.
\end{theorem}

We start by proving a variant of Theorem~\ref{balancedsep} for $(L,d,1)$-breakable graphs. 

\begin{theorem} \label{balancedsep_breakable}
Let $L,d$ be integers, and let $G$ be an $(L,d,1)$-breakable graph.
Let $w$ be a normal weight function on $G$.
Then there exist a clique $K$ in $G$  and
a set $Y(K) \subseteq V(G) \setminus K$ such that
\begin{itemize}
\item $|K| \leq d$,
\item $\alpha(Y(K)) \leq d^2L$
  and
\item $N[K] \cup Y(K)$ is a $w$-balanced separator in $G$.
\end{itemize} 
\end{theorem}

\begin{proof}
 Since $G$ is $(L,d,1)$-breakable, there exists  
 $X \subseteq V(G)$ with $|X| \leq d$
    such that for every component $D$ of $G  \setminus N[X]$,
    $w(D) \leq \frac{1}{2}$. For every pair $x,x'$ of non-adjacent
    vertices of $X$, let $Y(x,x')$ be a set with $\alpha(Y(x,x')) \leq L$  and 
    $Y(x,x') \cap \{x,x'\}=\emptyset$
    separating $x$ from $x'$ in $G$ (such a set exists since $G$ is 
    $(L,d,1)$-breakable). Now let
    $$Y=X \cup \bigcup_{x,x' \in X \text{ non-adjacent }}Y(x,x').$$
    Then $\alpha(Y) \leq \binom{d}{2} L +d \leq d^2L$.
    If $Y$ is a $w$-balanced separator of $G$, set
    $K=\emptyset$ and $Y(K)=Y$, and the theorem holds.
    Thus we may assume that there is a component $D$ of
    $G \setminus Y$ with $w(D) > \frac{1}{2}$.
    Let $K \subseteq X$
    be the set of vertices of $X$ with a neighbor in $D$.
Then $D \cap N(X) = D \cap N(K)$.
    Since  every two non-adjacent vertices of $X$ are separated by
    $Y$, it follows that $K$ is a clique. 
    
    We claim that $N[K] \cup Y$ is a $w$-balanced separator in $G$.
    Suppose not, and let $D'$ be  the component of $G \setminus (N[K] \cup Y)$ with $w(D') > \frac{1}{2}$. Then $D' \subseteq D$. But $D\cap N(X)= D\cap N(K) \subseteq Y$,
    and consequently $D' \cap N[X]=\emptyset$, contrary to the fact that $N[X]$ is a
    $w$-balanced separator in $G$. Thus, setting  $Y(K)=Y$,
    the theorem holds. 
\end{proof}

We are now ready to prove Theorem~\ref{smallsep}. Let us briefly describe the idea of the proof. Throughout, we have $Z_{i-1}$,  which is part of the separator we are building, and also a clique $L_{i-1}$ and cliques $K_1, \dots, K_{i-1}$ such that $N(K_j \cap L_{i-1})$ is a balanced separator in $G \setminus Z_{i-1}$. 

To achieve that, at  each step, we do one of the following. If
$L_i \cap K_i$ is empty, and then $Z_i$ is the balanced separator we are looking for. Otherwise,  we add at least one vertex $v$ to create $L_i$ from $L_{i-1}$
which means that $v$ has a neighbor in each previous $K_j$ (since $L_i$ remains a clique). 

Now our goal becomes controlling vertices with a neighbor in each $K_j$; we call them $Bad_{i-1}$ and we may assume they have big stability number. This tells us via Lemma \ref{TheLemma} that we can remove a set of small stability number, and for all remaining vertices, their neighbors in $Bad_{i-1}$ have stability number only a small fraction of $\alpha(Bad_{i-1})$. 

Theorem \ref{balancedsep_breakable} gives us a balanced separator, but it consists of a clique ($K_i$), its neighbors, and some other vertices ($Y(K_i)$). We add $K_i$ and $Y(K_i)$ to our separator; both have bounded stability number. Since $K_i$ is small, with $Bad_i = N(K_i) \cap Bad_{i-1}$, we have $\alpha(Bad_i) \leq \alpha(Bad_{i-1})/2$, which means that after logarithmically many steps, $Bad_i$ will be empty and the process terminates. 

It remains to build $L_i$. To do so, we find a small set of separators for each vertex $v \in K_i$ and its non-neighbors in $L_{i-1}$ and add it to $Z_i$. Now, for the big component $D'$ of $G \setminus Z_i$, the parts of $K_i$ and $L_{i-1}$ that attach to $D'$ are complete to each other, and it turns out that we can restrict ourselves to those parts for $L_i$. 

\begin{proof}

  Let $G$ be an $(L,d,r)$-breakable graph on $n$ vertices.
  Let $C(d)=100d^2$. 
  We define several sequences of subgraphs of $G$ and subsets of $V(G)$.
  Let $G_0=G$; let $K_0=Y(K_0)=L_0=Z_0=\emptyset$ and let $Bad_0=V(G)$.
We iteratively define $G_i, K_i, L_i, Z_i, Bad_i$ with the following properties: 
\begin{enumerate}[(I)]
\item \label{K}  If $i > 0$, then $K_{i} \cap Z_{i-1} = \emptyset$.
\item \label{smalljunk} $\alpha(Z_{i}) \leq C(d)L \cdot \lceil \frac{di}{r} \rceil i/2$.

 \item \label{monotone} $G_{i} = G \setminus Z_{i}$. 
\item \label{Bad} If $v \in G_{i}$ has a neighbor in $K_j$ for every 
$j \in \{1, \dots, i\}$, then $v \in Bad_{i}$.
\item \label{smallbad} $\alpha(Bad_{i}) \leq \frac{n}{2^{i}}$. 
 \item \label{smallLi} 
    $L_{i}$ is a clique and $|L_{i}| \leq di$.
 
\item \label{Liprop}
       For all $1 \leq j \leq i$, $Z_i \cup N(K_j \cap L_i)$ is a $w$-balanced separator in $G$.

    \end{enumerate}
 We proceed as follows. Suppose that we have defined 
  $G_{i-1}, K_{i-1},L_{i-1},Z_{i-1}, Bad_{i-1}$ satisfying the properties above. 
 If $Z_{i-1}$ is a $w$-balanced separator of $G$, we stop the construction.
 Otherwise, we
    construct  $G_i, K_i,L_i,Z_i, Bad_i$
    and show that the properties above continue to hold.
    If $\alpha(Bad_{i-1}) > 96d^2$, 
let $Z$ be obtained by applying Lemma~\ref{TheLemma} to $G$
with $Y= Bad_{i-1}$ and $c=2d$; then $\alpha(Z) \leq 4d$. 
If $\alpha(Bad_{i-1}) \leq 96d^2$, let $Z=Bad_{i-1}$.

In both cases, $\alpha(Z) \leq 96 d^2$.
Let $G'=G_{i-1} \setminus Z$.
    Let 
    $$w'(v)=\frac{w(v)} { \Sigma_{v \in V(G')}w(v)}.$$
    Then $w'$ is a normal weight function on $G'$, and
    for every $v \in G'$, $w'(v) \geq w(v)$.
        Let $K_i,Y(K_i)$ be as in Theorem~\ref{balancedsep_breakable} applied to
$G'$ and $w'$. It follows that $K_i$ is a clique of size at most $d$, and $K_i \subseteq G' \subseteq G \setminus Z_{i-1}$, so \eqref{K} holds (for $i$).

Let $v \in K_i$. By \eqref{smallLi} (for $i-1$),   $L_{i-1} \setminus N(v)$ can be partitioned into at most $\left\lceil \frac{d(i-1)}{r}\right\rceil$ cliques each of size at most $r$. 
Since $G$ is $(L,d,r)$-breakable, this implies that there
exists  a set $Z(v)$ with $\alpha(Z(v)) \leq \left\lceil \frac{d(i-1)}{r} \right\rceil L$, such that
$Z(v)$ separates $\{v\}$ from $L_{i-1} \setminus N(v)$ in $G$.
Let $Z'=\bigcup_{v \in K_i}Z(v)$; then
$\alpha(Z') \leq d\left\lceil \frac{d(i-1)}{r} \right\rceil L$.

Let $Z_i=Z_{i-1} \cup Z' \cup Z \cup K_i \cup Y(K_i)$.
Next, we have:
$$\alpha(Z_i \setminus Z_{i-1}) \leq d \left\lceil \frac{d(i-1)}{r} \right\rceil L+96d^2 + 1 + d^2L \leq C(d)\left\lceil \frac{di}{r} \right\rceil L/2.$$
It follows that $\alpha(Z_i) \leq C(d) \left\lceil \frac{di}{r} \right\rceil L \times i/2$, and
\eqref{smalljunk} holds.

Let $G_i=G_{i-1} \setminus Z_i = G \setminus Z_i$; now \eqref{monotone} holds.
Let $Bad_i=Bad_{i-1} \cap N_{G'}(K_i) \cap V(G_i)$; now \eqref{Bad} holds.  
Since $G' = G_{i-1} \setminus Z$ and either $Z = Bad_{i-1}$ or $Z$ is the set of all vertices $v$ such that $\alpha(N(v) \cap Bad_{i-1}) \geq \alpha(Bad_{i-1})/2d$ and by \eqref{smallbad} for $i-1$, we have that for every $v \in K_i$,
$\alpha(N_{G'}(v) \cap Bad_{i-1}) \leq \frac{n}{2^id}$.
It follows that $\alpha(Bad_i) \leq \frac{n}{2^i}$, and
\eqref{smallbad} holds for $i$.

 If $Z_i$ is a $w$-balanced separator in $G_{i-1}$, let $L_i=L_{i-1}$;
 now \eqref{smallLi} and \eqref{Liprop} hold. Thus we may assume not, and let 
  $D'$ be the maximal  connected subset of $G_i = G \setminus Z_i$ with $w(D') > \frac{1}{2}$.

 We now define $L_i$ and check that
\eqref{smallLi} and \eqref{Liprop} hold.
Since $N[K_i] \cup Y(K_i)$ is a $w'$-balanced separator in $G'$ and since $G_i \subseteq G' \setminus (K_i \cup Y(K_i))$,
it follows that $D' \cap N(K_i) \neq \emptyset$, and so $N(D') \cap K_i \neq \emptyset$.
Since $Z' \cap G_i =\emptyset$, and since for every $v \in K_i$, $Z(v) \subseteq Z'$ separates $v$ from $L_{i-1} \setminus N(v)$, it follows  that $N(D')  \cap K_i$ is complete
to $N(D') \cap L_{i-1}$; let 
$$L_i=(N(D') \cap L_{i-1}) \cup (N(D') \cap K_i)\mbox{.}$$
Then $L_i$ is a clique, and no vertex of $K_i \setminus L_i$
has a neighbor in $D'$. Since $L_i \setminus L_{i-1} \subseteq K_i$,
\eqref{smallLi} holds.

In order to prove  \eqref{Liprop}, let us consider first the case $j = i$. Suppose for a contradiction that $D$ is a component of $G \setminus (Z_i \cup N(K_i \cap L_i))$ with $w(D) > 1/2$. Then $D \subseteq D'$. Moreover, $K_i \cap L_i = N(D') \cap K_i$, and so $N_{G_i}(K_i \cap L_i) \cap D' = N_{G_i}(K_i) \cap D'$. Therefore, $D \subseteq D'\setminus N_{G_i}(K_i) \subseteq G_i \setminus N(K_i)$. However, since $N[K_i] \cup Y(K_i)$ is a $w'$-balanced separator in $G'$, and since $K_i \cup Y(K_i) \subseteq Z_i$, it follows that $w(D) \leq w'(D) \leq 1/2$. 

Next, we consider the case when $j < i$. By \eqref{Liprop} for $i-1$, we know that $Z_i \cup N(K_j \cap L_{i-1})$ is a $w$-balanced separator in $G$. Suppose for a contradiction that $D$ is a component of $G \setminus (Z_i \cup N(K_j \cap L_i))$ with $w(D) > 1/2$. Then $D \subseteq D'$. We have that $N(D') \cap K_j \cap L_{i-1} \subseteq L_i \cap K_j$, and therefore, $N_{G_i}(L_{i-1} \cap K_j) \subseteq N_{G_i}(L_i \cap K_j)$. It follows that $D \subseteq D' \setminus N(L_{i-1} \cap K_j) \subseteq G \setminus (Z_i \cup N(L_{i-1} \cap K_j))$, a contradiction as $Z_i \cup N(L_{i-1} \cap K_j)$ is a $w$-balanced separator in $G$.

 We have shown that properties
\eqref{K}--\eqref{Liprop} are maintained at each step of the
construction.

We can now complete the proof of the theorem.
It follows immediately from \eqref{smallbad}
that there exists $k  \leq 1+ \log n$ such that $Bad_k=\emptyset$.
We claim that $Z_k$ is  a $w$-balanced separator in $G$ (and in particular the
construction stops).
Suppose not.  Then the construction continues and the sets 
$G_{k+1}, K_{k+1}, L_{k+1}, Z_{k+1}, Bad_{k+1}$ are defined.
Also, there exists a component $D$  of $G \setminus Z_{k}=G_{k}$
such that $w(D)> \frac{1}{2}$. We apply 
\eqref{Liprop}  with $i=k+1$. Since $Z_k$ is not a $w$-balanced separator,
it follows that $N(L_{k+1} \cap K_{k+1}) \neq \emptyset$, and consequently there is a vertex $v \in L_{k+1} \cap K_{k+1}$. By \eqref{K} and \eqref{monotone}, it follows that $v \in G_k$. Since $L_{k+1} \cap K_j \neq \emptyset$   for
every $1 \leq j \leq k$ (again by \eqref{Liprop}), we deduce  from  \eqref{Bad} and \eqref{smallLi}  that $v \in Bad_k$, a contradiction.

Now by  \eqref{smalljunk} we have
$$\alpha(Z_{k+1}) \leq C(d) \left\lceil \frac{d(k+1)}{r} \right\rceil L \times (k+1)/2 \leq
C(d)\frac{d(\log n+2)}{r} (2+\log n)L,$$
as required.

\end{proof}

\section{The proof of Theorem~\ref{treealpha}} \label{sec:theend}

We start with a lemma. (There are many versions of this lemma; we chose
one with a simple proof, without optimizing the constants.)
\begin{lemma}\label{lemma:bs-to-treealpha}
  Let $G$ be a graph, let $c \in [\frac{1}{2}, 1)$, and let $d$ be a positive integer. If for every normal weight function $w$ on $G$, there is a
    $(c,w)$-balanced separator $X_w$ with $\alpha(X_w) \leq d$, then 
        the tree independence number of $G$ is at most $\frac{3-c}{1-c}d$. 
\end{lemma}

\begin{proof}
  We will prove that for every  set $Z \subseteq V(G)$ with
  $\alpha(Z) \leq \frac{2}{1-c}d$
  there is a tree decomposition $(T,\chi)$ of  $G$ such that
  $\alpha(\chi(v)) \leq \frac{3-c}{1-c}d$ for every $v \in T$, and
  there exists $t \in T$ such that $Z \subseteq \chi(t)$.

  The proof is by induction on $|V(G)|$. Observe that every induced subgraph
  of $G$ satisfies the assumption of the theorem.
  
  Let $Z \subseteq V(G)$ with $\alpha(Z) \leq \frac{2}{1-c}d$.
  Let $I$ be a stable set of $Z$ with $|I|=\alpha(Z)$.
  Define a function $w$ where $w(v)=\frac {1}{|I|}$ if $v \in I$,
  and $w(v)=0$ if $v \not \in I$. Then $w$ is a normal weight function on
  $G$. Let $X$ be a  $(c,w)$-balanced separator with $\alpha(X) \leq d$.
  Then $V(G) \setminus X =V_1  \cup \dots \cup V_q$, where $V_1, \dots, V_q$ are the components of $G \setminus X$, and $|I \cap V_i| \leq c|I|$ for $i\in \{1,\dots, q\}$. Let $i \in \{1, \dots, q\}$. 
  Define $Z_i=(Z \cap V_i) \cup X$. Since
  $|I \cap V_i| \leq c|I|$,
  it follows that $|I \cap (G \setminus V_i)| \geq (1-c)|I|$.
Since $\alpha(X) \leq d$, we have that $|I \cap (G \setminus (V_i \cup X))| \geq (1-c)|I|-d$.
  It follows that 
  $\alpha(Z \cap V_i) \leq c|I|+d$. Consequently,
  $$\alpha(Z_i) \leq \alpha(Z \cap V_i) +\alpha (X) \leq c|I|+2d \leq \frac{2c}{1-c}d + \frac{2(1-c)}{1-c}d = \frac{2}{1-c}d.$$

  Inductively,  for $i \in \{1, \dots, q\}$, there is a tree decomposition $(T_i,\chi_i)$ of
  $V_i \cup X$ such that
  $\alpha(\chi_i(t)) \leq \frac{3-c}{1-c}d$ for every $v \in T_i$, and
  there exists $t_i \in T_i$ such that $Z_i \subseteq \chi(t)$.
  Now let $T$ be obtained from the disjoint union of $T_1, \dots, T_q$
  by adding  a new vertex $t_0$ adjacent to $t_1, \dots, t_q$ (and with no other
  neighbors). Define $\chi(t)=\chi_i(t)$ for every $t \in T_i$,
  and let $\chi(t_0)=Z \cup X$. Since
  $$\alpha(Z \cup X) \leq \alpha(Z)+\alpha(X) \leq \frac{2}{1-c}d+d=
  \frac{3-c}{1-c}d,$$
  $(T,\chi)$ satisfies the conclusion of the theorem.
  \end{proof}

  Next we restate and prove  Theorem~\ref{treealpha}:

 \begin{theorem} \label{treealpha_2}
  There exists a constant $c$ such that for every integer $n>1$  every 
  $n$-vertex graph $G \in \mathcal{C}$ has tree independence number at most
  at most $c(\log n)^2$.
 \end{theorem}

 \begin{proof}
   Let $G \in \mathcal{C}$.
Let $d$ be as in Theorem~\ref{balancedsep}.
Let $C(d)$ be as in Theorem~\ref{smallsep}, and let
$c=165C(d)$. We may assume that $n \geq c$.
Let $r=d(2+\log n)$, and let $L=32 \log n$.
By Theorems~\ref{balancedsep} and \ref{H1H2logn}, and since every clique
is cooperative, it follows that $G$ is $(L,d,r)$-breakable.
Now  by Theorem~\ref{smallsep}, for every normal weight function $w$ on $G$,
there exists a $w$-balanced separator $Y$ in $G$ such that
$\alpha(Y) \leq C(d)  \frac{d (2+\log n)} {r}(2+\log n) L \leq 33 C(d) (\log n)^2$ (since $n>1$).
Now  Theorem~\ref{treealpha_2} follows from Lemma~\ref{lemma:bs-to-treealpha}.
\end{proof}

\section{Algorithmic consequences} \label{sec:alg}
Theorem~\ref{treealpha_2} implies quasi-polynomial time (namely $2^{(\log n)^{O(1)}}$ time) algorithms for a number of problems. 
In particular Dallard et al.~\cite{dms2} gave $n^{O(k)}$ time algorithms for \textsc{Maximum Weight Independent Set} and \textsc{Maximum Weight Induced Matching} assuming that a tree decomposition with independence number at most $k$ is given as input. 
Lima et al.~\cite{lima2024tree} gave an $n^{O(k)}$ time algorithm for \textsc{Maximum Weight Induced Forest} assuming that a tree decomposition with independence number at most $k$ is given as input. 
Dallard et al.~\cite{dfgkm} gave an algorithm that takes as input a graph $G$ and integer $k$, runs in time $2^{O(k^2)}n^{O(k)}$ and either outputs a tree decomposition of $G$ with independence number at most $8k$, or determines that the tree independence number of $G$ is larger than $k$.
Theorem~\ref{treealpha_2}, together with these results (setting $k = c \log^2 n$), immediately imply the following theorem. 

\begin{theorem}\label{thm:alg_iset}
\textsc{Maximum Weight Independent Set}, \textsc{Maximum Weight Induced Matching}, and  \textsc{Maximum Weight Induced Forest} admit algorithms with running time $n^{O((\log n)^3)}$ on graphs in $\mathcal{C}$.
\end{theorem}

It is worth mentioning that the $n^{O(k)}$ time algorithm of Dallard et al.~\cite{dms2} works for a slightly more general packing problem (see their Theorem 7.2 for a precise statement) that simultaneously generalizes \textsc{Maximum Weight Independent Set} and \textsc{Maximum Weight Induced Matching}. Thus we could have stated Theorem~\ref{thm:alg_iset} for this even more general problem. 

Lima et al.~\cite{lima2024tree} observed that the algorithm of Dallard et al.~\cite{dms2} can be generalized to a much more general class of problems. In particular they show that for every integer $\ell$ and CMSO$_2$ formula $\phi$, there exists an algorithm that takes as input a graph $G$ of tree independence at most $k$, and a weight function $w : V(G) \rightarrow \mathbb{N}$, runs in time $f(k,\phi,\ell) n^{O(\ell k)}$ and outputs a maximum weight vertex subset $S$ such that $G[S]$ has treewidth at most $\ell$ and $G[S] \models \phi$. This formalism captures all problems mentioned in Theorem~\ref{thm:alg_iset}, \textsc{Maximum Weight Induced Path}, recognition of many well-studied graph classes (including $\mathcal{C}$) and a host of other problems. We remark that their result (Theorem 6.2 of~\cite{lima2024tree}) is stated in terms of clique number rather than treewidth, however at the very beginning in the proof they show that in this context bounded clique number implies treewidth at most $\ell$ and then proceed to prove the theorem as stated here.
Unfortunately the algorithm of \cite{lima2024tree} does not give any meaningful results when combined with Theorem~\ref{treealpha_2}. The reason is that the function $f(k,\phi,\ell)$ bounding the running time of the algorithm is a tower of exponentials, which leads to super-exponential running time bounds even when $k = c \log^2 n$. 
We conjecture that the algorithm of \cite{lima2024tree} can be modified to run in time $(f(\ell, \phi)n)^{O(\ell k)}$, which is quasi-polynomial for every fixed $\ell, \phi$ when $k = O(\log^{O(1)} n)$. Such an improvement would immediately yield quasi-polynomial time algorithms for all of the problems mentioned above (and others, see~\cite{lima2024tree}) on graphs in $\mathcal{C}$.





\section{Acknowledgments}
Peter Gartland proved Theorem~\ref{balancedsep} independently  using a
different method. We thank him for sharing his proof with us, and for
many useful discussions. We also thank the anonymous referee for careful reading of our paper and for many helpful suggestions.

\bibliographystyle{abbrv}
\bibliography{ref}

\end{document}